\documentclass{gOMS2e}

\usepackage{epstopdf}
\usepackage{subfigure}

\theoremstyle{plain}

\newcommand{\sgn}{\mbox{{sgn}}}
\newcommand{\xista}{x^k_{\mbox{\tiny ISTA}}}
\newcommand{\defeq}{\stackrel{\rm def}{=}}
\DeclareMathOperator*{\argmin}{argmin}
\usepackage{algorithm}
\usepackage{algorithmic}
\usepackage{color}
\usepackage{url}
\newtheorem{thm}{Theorem}[section]
\newtheorem*{thm*}{Theorem}

\newtheorem{asu}[thm]{Assumption}

\theoremstyle{definition}

\theoremstyle{remark}

\begin{document}



\title{A Second-Order Method for Convex $\ell_1$-Regularized Optimization with Active Set Prediction}

\author{
\name{Nitish Shirish Keskar\textsuperscript{a},
 Jorge Nocedal\textsuperscript{a}, Figen \"{O}ztoprak\textsuperscript{b} and Andreas W\"{a}chter\textsuperscript{a}}
\affil{\textsuperscript{a}Northwestern University, USA;
\textsuperscript{b} Istabul Bilgi University}
}

\maketitle

\begin{abstract} We describe an active-set method for the minimization of an objective function $\phi$ that is the sum of a smooth convex function and an $\ell_1$-regularization term. A distinctive feature of the method is the way in which active-set identification and {second-order} subspace minimization steps are integrated to combine the predictive power of the two approaches. At every iteration, the algorithm selects a candidate set of free and fixed variables, performs an (inexact) subspace phase, and then assesses the quality of the new active set. If it is not judged to be acceptable, then the set of free variables is restricted and a new active-set prediction is made.
We establish global convergence for our approach, and compare the new method against the state-of-the-art code LIBLINEAR.
 \end{abstract}


\begin{classcode}
49M; 65K; 65H; 90C
\end{classcode}
\section{Introduction}
\label{intro}
The problem of minimizing a composite objective that is the sum of a smooth convex function and a regularization term has received much attention; see e.g. \cite{sra2011optimization,bach2012optimization} and the references therein. This problem arises in statistics, signal processing, machine learning and in many other areas of applications. In this paper we focus on the case when the regularizer is defined in terms of an $\ell_1$-norm, and  propose an algorithm that employs a recursive active-set selection mechanism designed to make a good prediction of the active subspace at each iteration.  This mechanism combines first- and second-order information, and is designed with the large-scale setting in mind.

The problem under consideration is given by
\begin{equation}    \label{prob}
        \min_{x \in \mathbb{R}^n} \ \phi(x) = f(x) + \mu \|x\|_1 .
\end{equation}
We assume that $f$ is a smooth convex function and $\mu >0$ is a fixed penalty parameter. 

The algorithm proposed in this paper is  different in nature from  the most popular methods proposed for solving problem \eqref{prob}. These include  first-order  methods, such as ISTA, SpaRSA and FISTA \cite{ista,sparsa,fista}, and  proximal Newton methods that compute a step by minimizing a piecewise quadratic model of \eqref{prob} using (for example) a coordinate descent iteration \cite{yuan2012improved,tseng2009coordinate,hsieh2011sparse,katya-tang2,istanbul,lee2012proximal,SchmidtProjectedNewton,olsen2012newton}.
{
The proposed algorithm also differs from methods that solve \eqref{prob} by reformulating it as a  bound constrained problem; for e.g. \cite{Wen:2010:FAS:1958654.1958661,koh2007interior, wen_convergence_2012, schmidt_fast, schmidt_thesis}. 
 
Our algorithm  belongs, instead, to the class of \emph{orthant-based methods} \cite{andrew2007scalable} that minimize a smooth quadratic model of $\phi$  on  a sequence of orthant faces of $\mathbb{R}^n$ until the optimal solution is found.  But unlike the orthant-based methods described in \cite{andrew2007scalable,dss} and the bound-constrained approaches in \cite{schmidt_thesis},
}
 every iteration of our algorithm consists of a \emph{corrective cycle} of orthant-face predictions and  subspace minimization steps. This  cycle is terminated when the orthant-face prediction is deemed to be reliable. After a trial iterate has been computed, a globalization mechanism accepts or modifies  it (if necessary) to ensure overall convergence of the iteration. 

The  idea of employing a correction mechanism for refining the selection of the orthant face was introduced in \cite{figi} for the case when $f$ is a convex quadratic function. That algorithm is, however, not competitive  with state-of-the-art methods in terms of CPU time because each iteration requires the exact solution of a subspace problem, which is expensive, {and because the orthant-face prediction mechanism is too liberal and can lead to long corrective cycles.}
%
%
These deficiencies are overcome in our algorithm, which introduces two key components. We employ an adaptive filtering mechanism that in conjunction with the corrective cycle yields an efficient prediction of zero variables at each iteration. We also design a strategy for solving, inexactly, the subproblems arising during each corrective step in a way that does not degrade the accuracy of the orthant-face prediction and yields important savings in computation. We show that the algorithm is globally convergent for strongly convex problems. Numerical tests on a variety of machine learning data sets suggest that our algorithm is competitive with a leading state-of-the-art code. 

The main features of our algorithm can also be highlighted by contrasting them with recently proposed proximal Newton methods for solving problem \eqref{prob}. The algorithms proposed by \cite{yuan2012improved,katya-tang2,hsieh2011sparse} and others first chose an active set of variables using first-order sensitivity information. The active variables are set to zero, and the rest of the variables are updated by minimizing a \emph{piecewise quadratic} approximation to \eqref{prob} given by
\begin{align}    \label{sqa}
   q^k(x) &= f(x^k) + (x- x^k)^T \nabla f(x^k)  + \frac{1}{2} (x - x^k)^T \nabla^2 f(x^k) (x - x^k) + \mu \| x\|_1. 
\end{align}
This minimization is performed inexactly using a randomized coordinate descent method. After a trial iterate is computed in this manner, a backtracking line search is performed to ensure decrease in $\phi(x)$.

The proximal Newton methods just outlined employ a very simple mechanism (the minimum norm subgradient) to determine the set of active variables at each iteration.  On the other hand, they solve the sophisticated lasso subproblem \eqref{sqa} that  inherits the non-smooth structure of the original problem and permits iterates to cross points of non-differentiability of $\phi(x)$. The latter property allows proximal Newton methods to refine the active set with respect to its initial choice. In contrast, our method invests a significant amount of computation in the identification of a working orthant face in $\mathbb{R}^n$, and then minimizes a simple smooth quadratic approximation of the problem on that orthant face,
\begin{align}    \label{subor}
   \bar{q}^k(x) &= f(x^k) + (x- x^k)^T \nabla f(x^k) + \frac{1}{2} (x - x^k)^T \nabla^2 f(x^k) (x - x^k) + \mu \zeta^T x,
\end{align}
where $\zeta$ is an indicator with values 0, 1 or -1, that identifies the orthant face. The working orthant is selected carefully, by verifying that the predictions made at each corrective step are realized. We do so because a simpler selection of the orthant face, such as that performed in the OWL method \cite{andrew2007scalable}, or the method described in \cite{dss} can generate poor steps in some circumstances.

Given that the two approaches (proximal Newton with coordinate descent solver and our proposed method) are  different in nature,  it is natural to ask if one of them will emerge as the preferred {second-order} technique for the solution of problem \eqref{prob}.  To answer this question we compared a MATLAB implementation of our approach  on {binary classification problems} with the well-known solver LIBLINEAR (written in C), based on CPU time. One of the main conclusions of this paper is that \emph{both approaches} have their strengths. 
Orthant-based methods have the attractive property that the subspace minimization can be performed by a direct linear solver or by an iterative method such as the conjugate gradient method, {which is efficient on a wide range of applications.}
On the other hand, the proximal Newton approach method is very effective on applications where the Hessian matrix is diagonally dominant (or nearly so).  In this case, the coordinate descent iteration is particularly efficient in computing an approximate solution of problem \eqref{sqa}.  Both approaches share the need for effective criteria for deciding when an approximate solution of the subproblem is acceptable. Most implementations of the proximal Newton method employ adaptive techniques (heuristics or rules based in randomized analysis), while our implementation employs the classic termination criteria based on the relative error in the residue of the linear system \cite{mybook}.

This paper is organized in 5 sections. In Section~\ref{algo} we outline the algorithm, paying particular attention to the  orthant-face identification mechanism. Section~\ref{scn:globalization} discusses the procedure by which we safeguard against poor steps and ensure global convergence of the algorithm. {In Section \ref{numerical}, we present a comparison of our algorithm against the state-of-the-art code LIBLINEAR for the solution of binary classification problems; some final remarks are made in Section~\ref{scn:final}.}

\section{The Proposed Algorithm}   \label{algo}

The algorithm exploits the fact that the objective function $\phi$ is smooth in any \emph{orthant face} of $\mathbb{R}^n$, which is defined as the intersection of an orthant in $\mathbb{R}^n$ and a subspace $\{ x: x_i=0, \, i \in I \subset \{1, \ldots n\}\}$.  

At every iteration, the algorithm identifies  an orthant face in $\mathbb{R}^n$ using sensitivity information, performs a minimization on that orthant face to produce a trial point, refines the orthant-face selection (if necessary), and repeats the process until the choice of the orthant face is judged to be acceptable. Upon termination of this cycle, a backtracking line search is performed where the trial points are projected onto the active orthant.

To describe the algorithm in detail, we introduce some notation.
Let $g(x)$ denote the minimum norm subgradient of the objective function \eqref{prob} at a point $x$. Thus, we have
\begin{align}
\label{eqn:MNSG}
 g_i(x) = &\begin{cases}
\nabla_i f(x) + \mu &\hspace{-1.8ex}\mbox{if } x_i>0 \mbox{ or } \left( x_i=0 \mbox{ and }  \nabla_i f(x) + \mu <0\right)\\
\nabla_i f(x) - \mu &\hspace{-1.8ex}\mbox{if } x_i<0 \mbox{ or } \left(x_i=0 \mbox{ and }  \nabla_i f(x) - \mu >0 \right)\\
0 &\hspace{-1.8ex}\mbox{otherwise}, \\
\end{cases} 
\end{align}
for $i=1, \ldots n$,  where
\[
       \nabla_i f(x) \defeq \frac{\partial f(x)}{\partial x_i}.
       \]
 At an iterate $x^k$, we define  three sets:
\begin{align}
\mathcal{A}^k &= \{ i | \, x_i^k =0 \mbox{ and } |\nabla_i f(x^k)| \leq \mu \} \label{aset} \\
\mathcal{F}^k &= \{ i | \, x_i^k \neq 0\} \label{fsure} \\
\mathcal{U}^k &=\{ i | \, x_i^k =0 \mbox{ and } |\nabla_i f(x^k)| > \mu \} . \label{unsure}
\end{align}
The variables in $\mathcal{A}^k$ are kept at zero (since the corresponding components of $g_i(x^k)$ are zero), while those in $\mathcal{F}^k$ are free to move. The remaining variables are in the set $\mathcal{U}^k$. The decision {of} which of these are allowed to move significantly impacts the efficiency of the algorithm. Using the \emph{selection mechanism} described below, we first create a partition of $\mathcal{U}^k$,
\begin{equation}  \label{char}
       {\cal U}^k={\cal U}_A \cup {\cal U}_F ,
 \end{equation}
 where the variables in $\mathcal{U}_A$ are fixed at zero and the variables in ${\cal U}_F$ are allowed to move. We then update  the active set as
  \begin{equation}   \label{active}
         \mathcal{A}^k \gets \mathcal{A}^k\cup \mathcal{U}_A,
 \end{equation}
 and compute a trial step  $d^k$ as the  (approximate) solution of the smooth quadratic problem
 \begin{align}    
       \min_{d \in \mathbb{R}^{n}} & \ \ \psi(d) =    d^Tg(x^k) 
                                          + \frac{1}{2} d^T  H^k d \nonumber \\
      \mbox{s.t.} & \ \ d_i=0, \ i \in \mathcal{A}^k , \label{aomodel}
\end{align}
where  $H^k = \nabla^2 f(x^k)$.  The trial iterate is defined as 
\[   {\hat x}^k= x^k + d^k. \]

We then start the \emph{corrective cycle} and check whether all variables in the set ${\cal U}_F$ moved as predicted; i.e.,  whether
\begin{equation} \label{goods}
        \sgn([{\hat x^k}]_i) = \sgn(-[g(x^k)]_i) \quad\mbox{for all   } i \in {\cal U}_F.
\end{equation}
Any variable $j \in \mathcal{U}_F$ for which this equality does not hold, is removed from the set $\mathcal{U}_F$ and added to $\mathcal{U}_A$.  The set $\mathcal{A}^k$ is then updated according to \eqref{active} and a new trial step is recomputed by solving \eqref{aomodel}. We repeat this corrective cycle until all predictions are correct and the trial point $\hat x^k$ satisfies \eqref{goods}.

The algorithm then performs a projected backtracking line search along $d^k$ 
to ensure that the resulting point {yields} a decrease in the piecewise quadratic model $q^k(x)$ defined in \eqref{sqa}. (We do not perform the line search on the objective function \eqref{prob} {as that is more expensive,} and the globalization mechanism described in Section~\ref{scn:globalization} only requires a decrease in $q^k(x)$.) 

At iteration $k$, we identify the current orthant face based on sensitivity information \eqref{eqn:MNSG} and define the vector $\zeta^k$ by
\begin{equation}  \label{zeta}
      \zeta^k_i = \begin{cases} \sgn([x^k]_i) \quad \ \ \mbox{if }  [x^k]_i\neq 0 \\ \sgn(-[g(x^k)]_i) \quad \mbox{if } [x^k]_i= 0 . \end{cases}  
\end{equation}
Let  $\mathcal{P}^k(x)$ be the projection operator {that projects $x \in \mathbb{R}^n$ onto the orthant defined by $\zeta^k$;} i.e., 
\begin{equation}
      \mathcal{P}^k_i(x) = \begin{cases} 
      x_i & \mbox{if } \sgn(x_i) = \sgn(\zeta_i^k) \\
      0 & \mbox{otherwise.}
      \end{cases}  
\end{equation}
We then search for the largest step size $\alpha \in \{2^0,2^{-1},2^{-2},\cdots \}$ such that 
$$q(x^k) \geq  q( \mathcal{P}^k(x^k + \alpha\cdot d^k)),$$
where $q$ is the non-smooth quadratic approximation given by \eqref{sqa}. 
Such a step size exists because $d^k$ is a descent direction for the smooth quadratic function $\bar{q}^k$ and because the trial point lies within the orthant defined by $\zeta^k$ for sufficiently small steps (see Theorem~\ref{thm:global_convergence} in the appendix).

Before giving a detailed description of the algorithm, we describe the  \emph{selection mechanism} that, at the beginning of each corrective cycle, defines the splitting \eqref{char} of the  set $\mathcal{U}^k$ into variables ${\cal U}_A $, that are kept at zero, and variables ${\cal U}_F$, that are allowed to move. 

At the start of the algorithm, we select a scalar $\eta \in (0,1)$ and set $| \mathcal{U}_F| = \tau^0 \defeq \lfloor \eta \times n \rfloor$; i.e., the cardinality of the set $\mathcal{U}_F$ is a fraction of the dimension of the problem. On subsequent iterations, we update the  parameter $\tau^k$  based on  its previous value $\tau^{k-1}$ and the number of iterations in the previous corrective cycle. If there were no corrections in the previous corrective cycle, we set $\tau^{k+1} = 2 \tau^k$ to allow more variables to change at the next outer iteration; otherwise  we keep the value of $ \tau^k$ unchanged. Since the number of variables in ${\cal U}_F$ cannot be larger than $|{\cal U}^k|$,  the actual size of ${\cal U}_F$ is given by
\[
     |\mathcal{U}_F |= \hat{\tau}^k \defeq \min(|\mathcal{U}^k|,\tau^k). 
\]
%
%
%
We use a greedy strategy to populate the sets $\mathcal{U}_A$ and  $\mathcal{U}_F$: we collect in ${\cal U}_F$ the $\hat \tau^k$ variables  in $\mathcal{U}^k$ with the largest components of the subgradient $|g(x^k)|$. Thus,  for any $i \in \mathcal{U}_F$ and $j \in \mathcal{U}_A$, we have $|g_i(x^k)|\geq |g_j(x^k)|$.

A formal description of the overall method is given in Algorithm~\ref{alg1}. 

\begin{algorithm}[H]
\caption{Preliminary Adaptive Orthant-Based Method}
\label{alg1}
\begin{algorithmic}[1]
\STATE Given $x^0 \in \mathbb{R}^n$, $L>0$, $\mu>0$,  ${\eta}\in (0,1)$.\\  Let  $\tau^0 = \lfloor \eta\times n \rfloor$.

\WHILE{$k=0,1,2,\cdots$ and stopping criterion not met}

\STATE {\it Active-Set Identification}: \label{alg:basic_start}
\begin{align*}
\mathcal{A}^k &= \{ i | (x_i)^k =0 \mbox{ and } |\nabla_i f(x^k)| \leq \mu \} \\
\mathcal{F}^k &= \{ i | (x_i)^k \neq 0\} \\
\mathcal{U}^k &=\{ i | (x_i)^k =0 \mbox{ and } |\nabla_i f(x^k)| > \mu \} 
\end{align*}
\STATE {\it Selection Mechanism}: \newline
{Compute $g(x^k)$ by \eqref{eqn:MNSG} and $\zeta^k$ by \eqref{zeta}.}

\STATE 
Set $\hat{\tau}^k \gets \min(|\mathcal{U}^k|,\tau^k)$.

\STATE Choose $\mathcal{U}_F, \mathcal{U}_A \subseteq \mathcal{U}^k$ such that $\mathcal{U}_F\cap\mathcal{U}_A=\emptyset$, $|\mathcal{U}_F|=\hat{\tau}^k$ and for any $i \in \mathcal{U}_F$ and $j \in \mathcal{U}_A$, $|g_i(x^k)|\geq |g_j(x^k)|$. 
\STATE Set $\mathcal{A}^k \gets \mathcal{A}^k\cup \mathcal{U}_A$.
\STATE Compute or update second-order approximation $H^k$.
\STATE {{\it Corrective Cycle}:} \newline
Set $V^k \gets \mathcal{U}_F$ and $j \gets 0$.
\WHILE{$V^k \neq \emptyset $}
\STATE $$d^k = \argmin_{d_i=0, i \in \mathcal{A}^k} d^T g(x^k) + \tfrac{1}{2} d^T H^k d$$
\STATE Set $\hat{x}^k \gets x^k+d^k$.
\STATE Set $V^k = \{ i \in \mathcal{U}_F \setminus \mathcal{A}^k | \sgn(\zeta_i^k) \neq \sgn(\hat{x}_i^k ) \}$.
\STATE 
Set $\mathcal{A}^k \gets \mathcal{A}^k\cup V^k$ and $j \gets j+1$.
\ENDWHILE \label{alg:corrective_loop_ends}
\IF{$j=1$}
\STATE Set $\tau^{k+1} = 2 \cdot \tau^k$.
\ENDIF
\STATE {\it Projected Line Search}: \newline Set $\alpha \gets 1$.
\WHILE{ $q(x^k) > q( \mathcal{P}^k(x^k + \alpha\cdot d^k)) $} \label{projls_loop_starts}
\STATE Set $\alpha \gets {\alpha/2}$.
\ENDWHILE \label{alg:basic_end}


%
\STATE Set $x^{k+1} =  \mathcal{P}^k(x^k + \alpha\cdot d^k)$.
\ENDWHILE
\end{algorithmic}
\end{algorithm}


In this paper we assume that the quadratic model \eqref{aomodel} employs exact Hessian information, i.e. $H^k = \nabla^2  f(x^k)$, and that we  perform an approximate minimization of this problem using the conjugate gradient method in the appropriate subspace of dimension $(n-|\mathcal{A}^k|)$; see e.g., \cite{mybook}. The matrix $H^k$ can  also be defined by quasi-Newton updates, specifically using the compact representations of limited-memory BFGS matrices \cite{ByrdNoceSch94}. Although we do not explore a quasi-Newton variant in this paper, we expect it to be  effective in many applications.

Our \textit{selection mechanism} for defining the splitting \eqref{char}  is motivated by the following considerations. If  all variables in ${\cal U}^k$ were allowed to move, the algorithm would have similar properties to the OWL  method \cite{andrew2007scalable}, whose performance is not uniformly successful (see Section~\ref{numerical}). Indeed, we observed a more reliable performance when the size of ${\cal U}_F$ is limited.  This also has computational benefits because the subproblem \eqref{aomodel} is less expensive to solve when the number of free variables is smaller (i.e., when the set ${\cal A}^k$ is larger).  On the other hand, in the extreme case $| {\cal U}_F |=1$, the algorithm resembles  a classical active-set method, which is not well-suited for large-scale problems.

These trade-offs are addressed by the dynamic strategy employed in steps 5 and 17 of Algorithm~\ref{alg1}. Initially, we choose ${\cal U}_F$ to be a small subset of ${\cal U}^k$  (by selecting $\eta$ to be small).  The algorithm increases the size of ${\cal U}_F$ in subsequent iterations if there is evidence that the current choice is too restrictive. As as indicator we observe the number of iterations in the previous corrective cycle. A small number of corrections (in our implementation this number is 1) suggests that the choice of ${\cal U}_F$ may be too conservative and the size of ${\cal U}_F$ is doubled at the next outer iteration. We have found that this selection mechanism leads to more gradual and controlled changes in the active set compared to other orthant-based methods like OWL and the method proposed in Section 5 of \cite{dss}. 


The projected backtracking line search in Algorithm~\ref{alg1} differs from that used by other orthant-based methods in that it is based on the quadratic model and not the objective function.  As in other orthant-based methods, the projection promotes sparsity in the iterates and provides some control for steps that leave the current orthant, outside of which the smooth approximation in \eqref{aomodel} is not valid. But in contrast to other orthant-based methods, such as OWL, the line search is not the main globalization mechanism in our algorithm, as described next. 

\section{Globalization Strategy}
\label{scn:globalization}
While  Algorithm~\ref{alg1} generally works  well in practice, it may fail (cycle) when the changes in the active set are not sufficiently controlled. By adding  a globalization mechanism we ensure that all iterates generated by the algorithm provide sufficient reduction in the objective function and converge to the solution.
Our mechanism employs the iterative soft-thresholding algorithm (ISTA)~\cite{ista,Daubechies:04} to generate a reference point.  Because the ISTA method enjoys a global linear rate of convergence on strongly convex problems, it provides a benchmark for the progress of our algorithm.

We modify Algorithm~\ref{alg1} as follows. The iterate computed in line 23 is now regarded as a trial iterate and denoted by $\hat x^k$. To decide if this point is acceptable we check whether it produces a lower function value than the ISTA step computed from the starting point of the iteration, $x^k$.
If so, we accept the trial point; otherwise, we search along the segment joining $\hat x^k$ and the ISTA point $\xista$ to find an acceptable point.
Given a Lipschitz constant $L$ for the gradient of $f$, the cost of computing the ISTA step is negligible since gradient information is already available at $x^k$. However, the evaluation of $\phi(\xista)$ incurs an additional cost. To get around this expense, we use the value of an upper quadratic approximation of $\phi$ at $\xista$ as a surrogate to $\phi(\xista)$. More specifically, assuming that $L$ is a Lipschitz constant of  $\nabla f$, we define the value of the surrogate function as
\begin{align}
\label{eqn:gamma}
\Gamma^k &= f(x^k) + \nabla f(x^k)^T (\xista-x^k)  + \tfrac{L}{2} \| \xista - x^k\|_2^2 + \mu \|\xista\|_1 .
\end{align}
The computation of $\Gamma^k$ requires only one inner product.
The complete version of the algorithm, including the globalization mechanism, is given in Algorithm~\ref{alg:complete}.
\begin{algorithm}[H]
\caption{Orthant-Based Adaptive Method (OBA) }
\label{alg:complete}
\begin{algorithmic}[1]
\STATE Given $x^0 \in \Re^n$, $L>0$, $\mu>0$,  $\eta \in (0,1)$ and $\epsilon>0$.\\ Let  $\tau^0 = \lfloor \eta\times n \rfloor$

\WHILE{$k=0,1,2,\cdots$ and stopping criterion not met}
\STATE Carry out steps 1 -- 22 of Algorithm \ref{alg1}.
\STATE Set $\hat{x}^{k} =  \mathcal{P}^k(x^k + \alpha\cdot d^k)$.

\STATE {\it Globalization}: \newline
Compute ISTA step at $x^k$ as $$\xista = \mathcal{S}_{\mu/L} (x^k - \tfrac{1}{L} \nabla f(x^k))$$
 where $\mathcal{S}_\alpha(x)$ is a component-wise operator defined as $ \mathcal{S}_\alpha (x)_i = \max\{| x_i | -\alpha,0\} \cdot \sgn (x_i)$. 
\STATE 
Set $\bar{d}^k \gets \hat{x}^k - \xista$ and $\bar{\alpha} \gets 1$.
\STATE Calculate $\Gamma^k$ using \eqref{eqn:gamma}.
\WHILE{$\phi(\xista+\bar{\alpha}\cdot \bar{d}^k) > \Gamma^k$} \label{alg:dogleg_sufficiency}
\STATE  Set $\bar{\alpha} \gets \bar{\alpha}/2$.
\IF{$\bar{\alpha} < \epsilon$}
\STATE Set $\bar{\alpha}\gets 0$.
\ENDIF
\ENDWHILE
\STATE Set $x^{k+1} =\xista+\bar{\alpha}\cdot \bar{d}^k$.
\ENDWHILE
\end{algorithmic}
\end{algorithm}

The following convergence result is proven in the appendix.

\begin{thm*}
Assume that $f$ is continuously differentiable and strongly convex and that $\nabla f$ is Lipschitz continuous. Then, the iterates  $\{x^k\}$ generated by Algorithm~\ref{alg:complete} converge to the optimal solution $x^\star$ of problem \eqref{prob} at a linear rate.
\end{thm*}

\section{Numerical Experiments}   \label{numerical}

In this section, we demonstrate the viability of our approach.
While our method  applies to any convex function with an additive $\ell_1$-regularizer, we focus on the specific problem of binary classification using logistic regression. This problem is well studied with theoretical guarantees and many data sets available of varying sizes, structures and fields of study. { Further, the results reported on this problem are representative of the performance of OBA  on other functions (including multi-class logistic regression, probit regression and LASSO) where similar trends are observed.} 
We direct the reader to \cite{friedman_regularization_2010} and the references therein for details regarding the function $f$ and the statistical justifications of this choice. 

The data sets chosen for comparison are listed in Table~\ref{table:datasets}. Synthetic is a randomly generated, balanced, non-diagonally dominant problem; the process for generating this problem is described in the appendix. Alpha is a data set from the Pascal Large Scale Learning Challenge \cite{alpha_dataset}. Both these data sets have been feature-{wise} normalized to $[-1,1]$. Details for the other data sets along with their preprocessing steps can be found in \url{http://www.csie.ntu.edu.tw/~cjlin/liblinear} and the references therein.

\begin{table}
\tbl{Data sets}
{
\begin{tabular}{rcc}
\toprule
Data set & number of data points & number of features  \\
\colrule
Gisette & 6000 & 5000 \\
RCV1 &  20242 & 47236 \\
Alpha & 500000& 500 \\
KDDA & 8407752 & 20216830 \\
KDDB & 19264097& 29890095 \\
Epsilon & 400000 & 2000     \\
News20 &  19996 & 1355191 \\
Synthetic &   5000 & 5000   \\
\botrule
\end{tabular}
}
\label{table:datasets}
\end{table}

 A variety of methods has been proposed for solving problem \eqref{prob}, and high-performance implementations of some of these methods are available. One of the most popular codes is newGLMNET \cite{yuan2012improved}, which is a C-implementation of a proximal Newton method and is a part of the LIBLINEAR package. Every iteration of this method identifies the active set as a subset of $\mathcal{A}^k$ as defined in \eqref{aset}, and then solves problem \eqref{sqa} inexactly using a randomized coordinate descent algorithm. The termination criterion for this inner loop is based on the $\ell_1$-norm of the minimum norm subgradient and adjusted by a heuristic as the iteration progresses.


 We implemented Algorithm \ref{alg:complete} in MATLAB, where we chose $\eta=0.01$ and $\epsilon=10^{-4}$. Subproblem \eqref{aomodel} is solved inexactly via the conjugate gradient algorithm. The termination criterion is based on the relative tolerance of the linear system: The inner loop stops as soon as the conjugate gradient iterate $p$ satisfies
 \begin{align*}
 \frac{\| H^k p + g^k \|_\infty}{\|g^k\|_\infty} \leq 0.1.
 \end{align*}
 
{We also compare with the OWL method \cite{andrew2007scalable}, as implemented by Schmidt \cite{schmidt_thesis}. The OBA algorithm with the selective-corrective mechanism removed is somewhat related to OWL. The primary differences between the two include the procedure of handling the active set constraints in the subproblem and the alignment step included in OWL. Specifically, OWL minimizes the quadratic model in \eqref{aomodel} over $\mathbb{R}^n$, aligns the search direction and then carries out a projected line search onto the active set.
 
 For all test problems, the regularization parameter $\mu$ was chosen through a $5$--fold cross validation. LIBLINEAR and OBA use the exact Hessian in defining the quadratic model \eqref{aomodel} and \eqref{sqa} while OWL uses a limited-memory BFGS approximation. Further, in order to solve singular problems, LIBLINEAR adds a small multiple (specifically, $10^{-12}$) of the identity to the Hessian and our algorithm uses the value of $10^{-8}$. LIBLINEAR employs a secondary mechanism to guard against singularity: it projects the result of the one dimensional optimization in the coordinate descent step onto the set $[-10,10]$. }

\subsection{Test Results}
\begin{figure}
\centering
\includegraphics[scale=0.7]{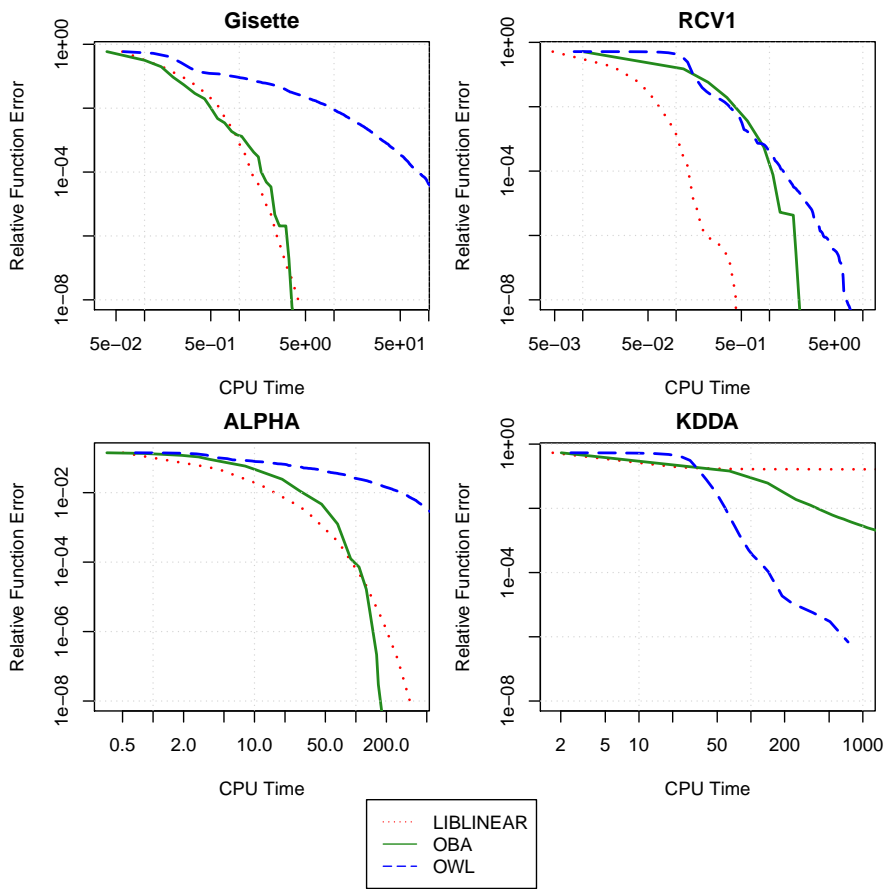}
\caption{  Relative error \eqref{ratio} in the objective function (vertical axis) versus CPU time -- Part 1}  
\label{fig:fig1}
\end{figure}

\begin{figure}
\centering
\includegraphics[scale=0.7]{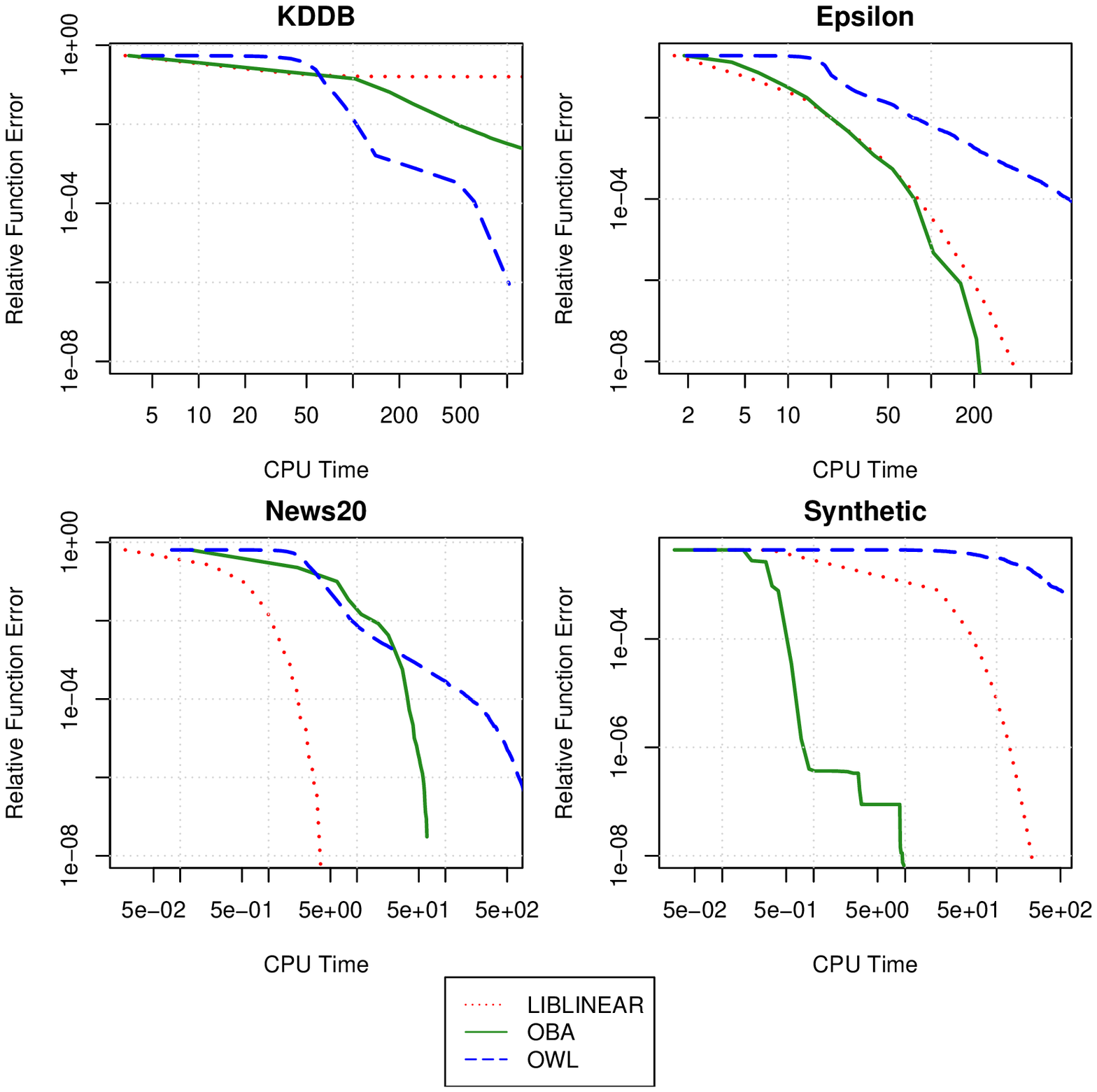}
\caption{  Relative error \eqref{ratio} in the objective function (vertical axis) versus CPU time -- Part 2}   
\label{fig:fig2}
\end{figure}

The comparison of the method proposed in this paper, Algorithm~2 (OBA), against LIBLINEAR and OWL  is presented in Figures \ref{fig:fig1} and \ref{fig:fig2}.  We  plot the relative function error defined as 
\begin{equation}  \label{ratio}
\frac{\phi(x^k)-\phi(x^\star)}{1+\phi(x^\star)}
\end{equation}
 against CPU time. The value of $\phi(x^\star)$ was obtained by running our  algorithm to a tight tolerance of $10^{-10}$ or until a time limit of 5000 CPU seconds was exceeded. The tolerance used corresponds to the one defined in \cite{istanbul}. The initial iterate for all methods was the zero vector. 

We can see that for RCV1 and News20, the performance of OBA  is inferior to LIBLINEAR; however, for KDDA, KDDB, Epsilon and Synthetic, the performance is superior. For Gisette and Alpha, the performance is roughly comparable irrespective of the value of the relative function error. OWL has gained a reputation as an algorithm which performs well but unreliably so. The experiments support this opinion. For problems like KDDA or KDDB, the performance of OWL is superior to both LIBLINEAR and OBA; however, for other problems like Synthetic, Alpha, Epsilon and Gisette, OWL fails to be competitive due to poor steps and rapid changes in working orthant faces.

 {We emphasize that the improved performance of the proposed method is driven  by the selective-corrective mechanism as opposed to the ISTA backup. The backup was \textit{never} required in the reported experiments for our method which, in contrast to other orthant-based methods, enjoys global convergence properties.}


\subsection{Sparsity}
\label{section:sparsity}
It is natural to ask whether an orthant-based method such as OBA is  as effective at generating sparsity in the solution as a proximal Newton method, such as LIBLINEAR. In  proximal Newton methods the non-smoothness of the original problem is retained in the subproblem \eqref{sqa}, and sparsity arises because the solution of the subproblem typically lies at points of non-differentiability.  In contrast, orthant-based methods solve a series of smooth problems that  have no tendency of inducing sparsity in the solution by themselves, but achieve it through the projection of the trial point onto the working orthant.

Both methods, LIBLINEAR and OBA, also promote sparsity through the definition of the active set at the beginning of each (outer) iteration, but the construction of the active set differs in the two methods.  LIBLINEAR fixes only a subset of the variables in the set $\mathcal{A}^k$ to zero; thus allowing some variables in ${\cal A}^k$  and all variables in the set ${\cal U}^k$ to move.  On the other hand, OBA  fixes all of the variables in ${\cal A}^k$ to zero and additionally fixes more variables in ${\cal U}^k$  through the selection mechanism and the corrective cycle. Therefore, the approach in LIBLINEAR can be considered  more liberal in that it releases more zero variables, while the approach in OBA can be regarded as more restrictive.  Nevertheless, OBA becomes increasingly more liberal as the iteration progresses because the selection mechanism allows the size of the set ${\cal U}_F$ to double under certain circumstances (see step 17 of Algorithm~\ref{alg1}).

In the light of these  algorithmic differences, it is difficult to predict the relative ability of the two methods at generating sparse solutions.  To explore this, we performed the following experiments using our data sets and recorded the sparsity in the solution.   LIBLINEAR was used to solve the problems with the tolerance of their stopping criterion set to $10^{-3}$ (which yielded better misclassification rates than the default value of $10^{-2}$), and OBA was then used to solve the problems to a similar accuracy in the objective function.  The results are presented  in Table~\ref{table:sparsity} and show that the two methods achieve similar values of sparsity, with OBA being somewhat more effective.

\begin{table}
\tbl{Percentage of Zeros in Solution}
{
\begin{tabular}{crr}
\toprule
Data set   & LIBLINEAR & OBA \\
\colrule
Gisette   & 88.92    & \textbf{90.52}     \\ 
RCV1      & 97.62   & \textbf{97.65}     \\ 
Alpha     & \textbf{5.60}    & 5.20   \\ 
KDDA    & 98.43   & \textbf{98.71}      \\ 
KDDB      & 97.11   & \textbf{97.87}     \\
Epsilon   & 44.60   & \textbf{69.20}     \\ 
News20    & \textbf{99.60}   & 99.37    \\ 
Synthetic & 56.86  & \textbf{58.82}     \\ 
\botrule
\end{tabular}
}
\label{table:sparsity}
\end{table}

\subsection{Conjugate Gradients versus Coordinate Descent}
{ Let us now focus on the methods used for solving the subproblems that incorporate second-order information about the objective function.} It is natural to employ the conjugate gradient (CG) method in OBA,  given that the subproblem \eqref{subor} is smooth and that the CG method is an optimal Krylov process that can exploit problem structure effectively. An alternative to the CG method is the randomized coordinate descent algorithm, which has gained much popularity in recent years \cite{friedman_regularization_2010,nesterov2012efficiency,richtarik2014iteration}

A drawback of  coordinate descent for smooth unconstrained optimization is that it can be slow when the Hessian is not diagonally dominant. We experimented with a coordinate descent solver for the subproblem in OBA and found that its overall performance is inferior to that of the CG method.

The situation is quite different in a proximal Newton method where the subproblem is  non-smooth. In that case, it is easy to compute the exact minimizer of  \eqref{sqa} along each coordinate direction, thereby dealing explicitly with the non-differentiability of the original problem. Since this one dimensional minimization may return zero as the exact solution, the proximal coordinate descent method provides an active-set identification mechanism for the overall algorithm.  Thus, although sensitivity to the lack of diagonal dominance may still be present, it is of a lesser concern due the benefits of its active set identification properties. Furthermore, applications in text classification and other areas often lead to problems with Hessians that are diagonally dominant \cite{Greene:2006icml}.

This discussion motivates us to look more closely at the issue of diagonal dominance and its effect on the two methods.   In order to quantify the level of diagonal dominance, we use the metric  employed, for example, in \cite{wright2014coordinatedescent}. Given any  symmetric matrix $A$, we define the level of diagonal dominance 
of $A$ as
\begin{align}
\label{eqn:wright_equation}
\mathcal{D}(A) = \frac{\max_i \|A_i\|_2}{\max_i |A_{ii}|},
\end{align}
where $A_i$ denotes the $i^{th}$ column of $A$ and $A_{ii}$ denotes the $i^{th}$ diagonal element of $A$. The smaller the value of $\cal D$, the closer is $A$ to being diagonally dominant. In Table \ref{table:diagonal_dominance}
we report the values of $\mathcal{D}\left(\nabla^2 f(x^0)\right)$ for all data sets in Table~\ref{table:datasets}. 


\begin{table}[h]
\caption{The value of $\mathcal{D}(\nabla^2 f(x^0))$ as defined in \eqref{eqn:wright_equation}}
\label{table:diagonal_dominance}
\begin{center}

    \begin{tabular}{| c | r |}
    \hline
    Data set   & $\mathcal{D}(\nabla^2 f(x^0))$ \\ \hline \hline
    Gisette   &      57.99             \\ \hline
    RCV1      & 1.88                  \\ \hline
    Alpha     & 9.93                   \\ \hline
    KDDA    & 1.80                   \\ \hline
    KDDB      &     1.50                   \\ \hline
    Epsilon   &      5.55              \\ \hline
    News20    & 3.29                   \\ \hline
    Synthetic & 69.42                 \\ \hline
    \end{tabular}

\end{center}
\end{table}

Let us begin by considering problem Synthetic, which was specifically constructed to have a high value of $\mathcal{D}$ (see Appendix \ref{reproducible} for details). We observe from Figure~\ref{fig:fig2} that LIBLINEAR performs poorly  compared to OBA.  This may be an indication that proximal Newton methods are sensitive to a lack of diagonal dominance.
In fact, by altering this problem so that $\cal D$ increases, the performance of LIBLINEAR deteriorates. The text classification tasks (RCV1 and News20), which are empirically observed to be diagonally dominant \cite{Greene:2006icml}, have low values of $\mathcal{D}$ and indeed, LIBLINEAR converges quickly\footnote{Interestingly, the LIBLINEAR website and manual convey that LIBLINEAR is known to perform well on document classification tasks but not necessarily on others.}. However, LIBLINEAR also performs well on problem Gisette for which $\cal D$ is large and poorly on KDDB for which $\cal D$ is low.  An examination of the rest of the results prevents us from establishing a clear correlation between the value of $\cal D$ and the relative performance of the two methods.  We conclude that in $\ell_1$-regularized problems the {adverse effects} of diagonal dominance appear to be less pronounced than for smooth optimization.  Other factors such as the frequency of orthant changes and the inexactness in the subproblem solution may also play a crucial role in explaining the performance differences.  The identification of problem characteristics that determine which method performs better for a given instances requires further investigation.

\section{Final Remarks}\label{scn:final}
In this paper, we  presented a second-order algorithm for solving convex $\ell_1$-regularized problems. At each iteration,  the algorithm tries to predict the orthant face containing the solution,  solves a smooth quadratic subproblem on this orthant face, and then invokes a corrective cycle that greatly improves the efficiency and robustness of the algorithm. We globalized the method by using the ISTA step as a reference for the desired progress.  This enabled us to prove a linear convergence rate of the iterates for strongly convex problems. {The ISTA backup is rarely used in practice (and never in the reported experiments) and thus, our theoretical result applies to a very robust method that invokes the safeguarding very rarely. This globalization procedure is analogous to  a Newton trust-region method where the underlying method is known to be very effective but convergence can only be proved by overcoming pathological situations with a first-order Cauchy step.} Numerical experiments for logistic regression data sets show that our algorithm is competitive in terms of CPU time with the LIBLINEAR C-code, even though our implementation is in MATLAB. The algorithm is also effective in generating sparse solutions quickly.
Overall, our experiments indicate that orthant-based methods are a viable alternative to proximal Newton methods.


\nocite{nesterov2004}
\bibliographystyle{gOMS}
\bibliography{OBA_references}

\appendices
\section{Convergence Analysis}   \label{global}
\setcounter{equation}{0}
Recall that we wish to solve the problem 
\begin{equation*}    
        \min_{x \in \mathbb{R}^n} \ \phi(x) = f(x) + \mu \|x\|_1.
\end{equation*}
For the purpose of our analysis, we make two assumptions:

\begin{asu}
\label{asm:stronglyconvex}
The function $f$ is in $C^1$ and strongly convex with parameter $\lambda>0$; i.e., for any $x,y \in \mathbb{R}^n$ and $t\in [0,1]$:
\begin{align}
f(t x+(1-t)y) & \leq  t f(x)+(1-t)f(y)-\frac{1}{2}\lambda t (1-t)\|x-y\|_{2}^{2}.
\end{align}
\end{asu}
As shown in Nesterov (2004), for continuously differentiable functions, this assumption is equivalent to
\begin{align}
f(y) \geq f(x) + \nabla f(x)^T(y-x) + \frac{\lambda}{2} \| y - x \|_2^2 \qquad \text{for all }x,y \in \mathbb{R}^n.
\label{eqn:strongconvexity_differentiable}
\end{align}

\begin{asu}
\label{asm:lipschitz}
The gradient of $f$ is Lipschitz continuous with constant $L>0$; i.e., for any $x,y \in \mathbb{R}^n$,
$$\| \nabla f(x) - \nabla f(y) \|_2 \leq L \| x - y\|_2 $$
\end{asu}

The first theorem shows that the algorithm is well-defined.
\begin{thm}
\label{thm:finite_termination}
The backtracking projected line search (steps  20--22 of Algorithm \ref{alg1}) terminates in a finite number of iterations. 
\end{thm}
\begin{proof}
Consider the $k^{th}$ iteration of Algorithm \ref{alg1}. For notational simplicity, we drop the iteration index and denote the iterate as $x$, the direction obtained after the corrective loop (steps 10--15) as $d$, and the smooth and non-smooth quadratic approximations as $\bar{q}$ and $q$, respectively. Along the same lines, let $\cal A$ and $\cal U$ be the active and unsure sets during this iteration.

We first show that there exists an $\bar{\alpha}>0$ such that for any $\alpha>0$ with $\alpha \leq\bar{\alpha}$, we have $\mathcal{P}(x+\alpha d) = x+\alpha d $. 
Let $\mathcal{I}_{1}=\{i\in\{ 1,2,3,\cdots,n\}:x_i\neq 0\}\}$ and $\mathcal{I}_{2}=\{i\in\{ 1,2,3,\cdots,n\}:x_i= 0\}\}$, and let $\bar{\alpha}>0$ such that $\bar{\alpha}<\big|\tfrac{x_{i}}{d_{i}}\big|$
for all $i\in\mathcal{I}_{1}$ with $d_i \neq 0$. 

Let $\alpha>0$ be such that $\alpha \leq\bar{\alpha}$ and $\zeta$ be defined in~\eqref{zeta}. We consider two cases:
\begin{itemize}

\item \textbf{Case 1: $i\in\mathcal{I}_{1}$}

Because $\alpha\leq\bar{\alpha}<\big|\tfrac{x_{i}}{d_{i}}\big|$, it is
clear that $\sgn(x_i+\alpha d_i)=\sgn(x_i)=\zeta_i$, and therefore $\mathcal{P}(x+\alpha d) = x+\alpha d $.
\item \textbf{Case 2: $i\in\mathcal{I}_{2}$}

By definition, $x_{i}=0$. If $i\in\mathcal{A}$ in step 19 of Algorithm 1, $d_i=0$ and $\sgn(x_i+\alpha d_i)=\sgn(x_i)=\zeta_i$.  Otherwise,  $i \in \mathcal{U}_F\subseteq \mathcal U$, and therefore $\sgn(-g_{i})=\zeta_i \in \{-1,1\} $. Assume that $\zeta_{i}=1$. Thus, $\sgn(-g_{i})=1$, and since $V^k=\emptyset$ and $i\in \mathcal U_F$, $\sgn(x^k+d^k)=\sgn(d^k)=1$, so $d_{i}> 0$ which in turn implies $\alpha d_i >0$.
The same conclusion can be made if $\zeta_{i}=-1$. 
Thus, $\mathcal{P}(x_i + \alpha d_i ) =\mathcal{P}(\alpha d_i ) = \alpha d_i = x_i + \alpha d_i $.

\end{itemize}

Because $d$ is a minimizer of $\bar{q}(x+d)$ in some subspace, we have $\bar{q}(x+\alpha d) \leq \bar{q}(x)$ for sufficiently small $\alpha\leq \bar\alpha$. Then, $\mathcal{P}(x+\alpha d) = x+\alpha d$, i.e., $x+\alpha d$ is in the same orthant as $x$, and therefore ${q}(x+\alpha d) = \bar {q}(x+\alpha d) \leq \bar{q}(x) = q(x)$.  As a consequence, the termination condition in the while-loop is satisfied after a finite number of iterations.


\end{proof}

We now show that by ensuring that $\phi$ at the new iterate is no larger than the majorizing function $\Gamma^k$, we can establish linear convergence. 
\begin{thm}
\label{thm:global_convergence}
Suppose that Assumptions \ref{asm:stronglyconvex} and \ref{asm:lipschitz} hold. Then, the iterates  $\{x^k\}$ generated by Algorithm~\ref{alg:complete} converge to the optimal solution $x^\star$ of problem \eqref{prob} at a linear rate.
\end{thm} 
\begin{proof}

\label{proof:global_convergence}

Consider the $k^{th}$ iteration of Algorithm \ref{alg:complete}. For notational simplicity, let us drop the iteration index and denote the minimum norm subgradient as $g$, the Hessian approximation as $H$, and the iterate as $x$.
Further, as is well known, the ISTA point $\xista$, computed in step 5 of Algorithm 2, is the minimizer of a proximal approximation of $\phi(x)$,
\begin{align}
\label{eqn:ista_defn}
\xista & = \arg\min_{y}f(x)+(y-x)^{T}\nabla f(x)+ \frac{L}{2}\|y-x\|_2^{2}+\mu\|y\|_{1}.
\end{align}
Because of Assumption \ref{asm:lipschitz}, for any $z_1,z_2 \in \mathbb{R}^n$,
\begin{align}
f(z_2) &\leq f(z_1) + \nabla f(z_1)^T (z_2-z_1) + \frac{L}{2} \| z_2-z_1\|_2^2.
\end{align}
In particular, by setting $z_1=x$, $z_2=\xista$, we get
\begin{eqnarray}
\phi(\xista) & = & f(\xista) + \mu\|\xista\|_{1}  \nonumber \\
&\leq & f(x)+\nabla f(x)^{T}(\xista-x)+\frac{L}{2}\|\xista-x\|_2^{2}+\mu\|\xista\|_{1} \equiv  \Gamma^k. \label{eqn:satisfaction_ista}
\end{eqnarray}

Let us denote the point obtained as a consequence of the globalization mechanism, which will be the new iterate, as $x^+$. This corresponds to $x^{k+1}$ in step 14 of Algorithm \ref{alg:complete}. Realize that the loop in steps 8--13 of Algorithm 2 terminates finitely because once $\bar{\alpha}$ drops to a value below $\epsilon$, it is set to $0$ and then $\phi(\xista+\bar{\alpha}\bar{d}) = \phi(\xista)$ and the sufficiency condition (step 8 of Algorithm \ref{alg:complete}) is trivially satisfied by \eqref{eqn:satisfaction_ista}. 

By design, our algorithm generates the point $x^{+}$ such that
\begin{align}
\phi(x^{+})\leq f(x)+\nabla f(x)^{T}(\xista-x)+\frac{L}{2}\|\xista-x\|^{2}+\mu\|\xista\|_{1}.
\label{eqn:xplus_ista_bound}
\end{align}
Combining this equation with the fact that $\xista$ is the minimizer in objective \eqref{eqn:ista_defn}, we have for
any $d\in \mathbb{R}^n$ and $y=x+\frac{\lambda}{L}d$  that
\begin{align*}
\phi(x^{+}) & \leq  f(x)+\nabla f(x)^{T} \left( \frac{\lambda}{L}d \right)+\frac{L}{2} \left\lVert \frac{\lambda}{L}d \right\lVert_2 ^{2}+\mu \left\lVert x+\frac{\lambda}{L}d \right\lVert _{1}\\
 & \leq  \phi \left(x+\frac{\lambda}{L}d \right) -\frac{\lambda}{2} \left\lVert\frac{\lambda}{L}d \right\lVert_2^2 + \frac{L}{2}\left\lVert \frac{\lambda}{L}d\right\lVert^{2}_2 \\
 & =  \phi \left(x+\frac{\lambda}{L}d \right)+\frac{\lambda^{2}}{2L}\left(1-\frac{\lambda}{L} \right)\|d\|_2^{2},
\end{align*}
where the second inequality follows from \eqref{eqn:strongconvexity_differentiable} with $y=x+\frac{\lambda}{L}d$.
In particular, we can set $d$ to be $ x^\star -x$ and obtain
\begin{align}
\phi(x^{+})  &\leq  \phi \left(x+\frac{\lambda}{L}(x^{\star}-x) \right)+\frac{\lambda^{2}}{2L}\left(1-\frac{\lambda}{L} \right)\|x^{\star}-x\|_2^{2}. \label{eqn:midwayproof}
\end{align}
Using Assumption \ref{asm:stronglyconvex} and the convexity of the $\ell_1$-norm, we have, for any $z_1,z_2 \in \mathbb{R}^n$ and $t\in[0,1]$, 
\begin{equation*}
\phi(tz_1+(1-t)z_2)  \leq  t\phi(z_1)+(1-t)\phi(z_2)-\frac{1}{2}\lambda t (1-t)\|z_1-z_2\|_{2}^{2}.
\end{equation*}
Setting $z_1=x$, $z_2=x^{\star}$, and $t= \left(1-\frac{\lambda}{L}\right)$, we get
\begin{equation*}
\phi \left( x +\frac{\lambda}{L}(x^{\star}-x) \right)  \leq  \frac{\lambda}{L}\phi(x^{\star})+\left( 1-\frac{\lambda}{L} \right) \phi(x)-\frac{\lambda^2}{2L} \left( 1-\frac{\lambda}{L} \right) \|x^{\star}-x\|_{2}^{2}.
\end{equation*}
Combining this result with \eqref{eqn:midwayproof}, we get
\begin{eqnarray*}
\phi(x^{+}) & \leq & \frac{\lambda}{L}\phi(x^{\star})+\left(1-\frac{\lambda}{L} \right)\phi(x)-\frac{\lambda^2}{2L}\left(1-\frac{\lambda}{L} \right)\|x^{\star}-x\|_{2}^{2} +\frac{\lambda^2}{2L}\left(1-\frac{\lambda}{L}\right)\|x^{\star}-x\|_2^{2}\\
 &  = & \phi(x^{\star})+\left( 1-\frac{\lambda}{L} \right) (\phi(x)-\phi(x^{\star})),
\end{eqnarray*}
and therefore
$$\phi(x^{+})-\phi(x^{\star})  \leq  \left( 1-\frac{\lambda}{L} \right) (\phi(x)-\phi(x^{\star})).$$
By reintroducing the iteration index and using recursion, we have that
$$\phi(x^{k})-\phi(x^{\star})  \leq  \left( 1-\frac{\lambda}{L} \right) ^k (\phi(x^0)-\phi(x^{\star}))$$
as required. 

\end{proof}
\section{Reproducible Research}   \label{reproducible}
The MATLAB code used to generate the ``Synthetic" problem is presented below. Given a dimension \texttt{n}, we use the following code snippet to generate the vector of labels (denoted by \texttt{y})  and the data matrix (denoted by \texttt{X}).

\begin{verbatim}
y=-1+(rand(n,1)>0.5)*2;
X = rand(n,n);
X = X + X';
mineig = min(eig(X));
if(mineig<0)
 X = eye(size(X))*mineig*-2+X;
end
X = chol(X);
\end{verbatim}

\end{document}